\documentclass[12pt]{amsart}
\usepackage{amscd, amsthm}
\usepackage{amsmath,amsfonts,amsthm,epsfig,latexsym,
graphicx,amssymb,psfrag}
\usepackage[english]{babel}
\usepackage{xcolor}

\newtheorem{prop}{Proposition}[section]
\newtheorem{theo}{Theorem}
\newtheorem{lemm}{Lemma}[section]

\newtheorem{rema}{Remark}
\newtheorem{clai}{Claim}
\newtheorem{claim}{Claim}[section]

\newtheorem{propri}{Properties}[section]

\newtheorem{proper}{Property}

 \newtheoremstyle{defi}
 {3 pt}{3 pt}{\rm}{3 pt}{\bf}{\hskip -0.5 pt . }{0 pt}
 {\thmname{#1}\thmnumber{#2}\thmnote{\textnormal(#3)}}
 \theoremstyle{defi}
 \newtheorem{defi}{Definition }[section]

\def\inf{\mathop{\rm Inf}} \def\infi{\mathop{\rm inf}}
\def\sup{\mathop{\rm Sup}} \def\supr{\mathop{\rm sup}}  \def\du{\mathop{\sqcup}}
\begin{document}

\title[.]{\bf Distortion in groups of Affine Interval Exchange transformations.}
\author{Nancy Guelman, Isabelle Liousse}

\address{ {\bf    Nancy  GUELMAN}, IMERL, Facultad de Ingenier\'{\i}a,  {Universidad de la Rep\'ublica,}{ C.C. 30, Montevideo, Uruguay.} \emph {nguelman@fing.edu.uy }. }

\address{{\bf    Isabelle LIOUSSE}, UMR CNRS 8524, Universit\'{e} de Lille1 ,59655 Villeneuve d'Ascq C\'{e}dex,   France.  \emph {liousse@math.univ-lille1.fr}. }

\begin{abstract}
\noindent  In this paper, we study distortion in the  group $\mathcal A$  of Affine Interval Exchange Transformations (AIET). We prove that any distorted element $f$ of $\mathcal A$, has an iterate $f^ k$ that is conjugate by an element of  $\mathcal A$ to a product of infinite order restricted rotations, with pairwise disjoint supports.
As consequences we prove that no Baumslag-Solitar group, $ BS(m,n)$ with $\vert m \vert \neq \vert n \vert $, acts faithfully by elements of $\mathcal A$; every finitely generated nilpotent group of $\mathcal A$  is virtually abelian  and there is no distortion element in $\mathcal A_{\mathbb Q}$, the subgroup of $\mathcal A$ consisting of  rational AIETs.
\end{abstract}

\maketitle

\section{Introduction.}

 In the recent years, notions of distortion have attracted the interest of many people working on geometric group theory as well as rigidity theory (see \cite{Fra} for a survey).

On one hand,  some results established the existence of distorted elements in transformations groups. For instance, D. Calegari and M. Freedman, in \cite{CF}, showed that all homeomorphisms of spheres are distorted.
 Moreover, in the case of the unit circle, they proved that
 every irrational Euclidean rotation is distorted inside the group of $C^{2-\varepsilon}$-diffeomorphisms for any $\varepsilon > 0$. Requiring smoothness, Avila proved in \cite{Av} that irrational rotations are
distorted  in $ {\textit {Diff }}^{\infty} (\mathbb S^1)$. In higher dimensions, Militon (see \cite{Mil}, Theorem 1) showed that irrational translations of the
$d$-dimensional torus are distorted in $ {\textit {Diff }}^{\infty} (\mathbb T^d)$.

On the other hand, a significant consequence of non existence of distortion is the proof of the Zimmer conjecture in dimension 2: ''any action of $SL(3, \mathbb Z)$  by area preserving diffeomorphisms on a surface, has finite image''.
For instance, Polterovich (\cite{Pol}) and Franks-Handel(\cite{FH}) proved that $ {\textit {Diff }} ^{ 1}_{\mu}(\Sigma ^2)$ does not contain distortion, where $ \mu $ is a full support measure on a compact surface $\Sigma ^2$.

Novak (in \cite{No})  proved that there is no element of distortion in the  group of Intervals Exchange Transformations: bijections of the unit interval that are piecewise increasing and isometric.

In this work, we deal with a closely related problem, namely the existence of distortion elements  inside the group of AIETs: Affine  Intervals Exchange Transformations, denoted by $\mathcal A$. Roughly speaking an AIET is a bijection of the unit interval that is  increasing and affine on a finite number of intervals.  If the endpoints of these intervals and their images are rational points, then the AIET is called a rational AIET. The set of rational AIETs is a subgroup of $\mathcal A$  and it is denoted by  $\mathcal A_{\mathbb Q}$.

Finitely generated groups of AIETs
have provided several algebraically
interesting groups, as for instance the classical Thompson's groups $F$, $T$ and $V$ \cite{CFP} as well as some of
their generalized versions \cite{Hi}, \cite{BS}, the fundamental groups of orientable surfaces \cite{Ghy1}, the modular
group \cite{Nav3}, wreath products of the form $\mathbb Z\wr  \mathbb Z \wr  ... \wr  \mathbb Z $ (\cite{Ble}, \cite{Bri}, \cite{Nav1}, \cite{Nav2}) etc.

\medskip

Our main result proves that  most elements of $\mathcal A$ are in fact undistorted.

\begin{theo}\label{TH}  For every distorted element $f$ of $\mathcal A$, there exists an integer $k>0$, such that $f^ k$ is conjugate by an element of  $\mathcal A$ to a
product of infinite order restricted rotations, with pairwise disjoint supports. \end{theo}

This description of distorted elements of  $\mathcal A$  enables us to prove the following statements.

\begin{theo}\label{BS}
The Baumslag-Solitar groups $BS(n,m)=<a,b \ \vert  \ ba^nb^{-1}=a^m>$ with $m$, $n$ integers and $\vert m\vert \not=\vert n\vert$ do not act faithfully via elements of $\mathcal A$.
\end{theo}

\begin{theo}\label{NIL}
Every torsion free nilpotent subgroup of $\mathcal A$ is  abelian and every finitely generated nilpotent subgroup of $\mathcal A$ is virtually abelian.
\end{theo}

As corollary, the Heisenberg group and thereby $SL(3,\mathbb Z)$ do not act faithfully via elements of $\mathcal A$.

\smallskip

Last two theorems were proved by Higman for Thompson's group $V$ (see \cite{Hi} and \cite{Ro}, chapter 2).

\begin{theo}\label{ratio}
There is no distortion elements in $\mathcal A_{\mathbb Q}$.
\end{theo}

This theorem  extends to all groups of rational AIET,  results of Burillo-Cleary-R\"over (see \cite{BCR}) and Hmili-Liousse (see \cite{HL}) on non existence of distortion in Thompson's groups $V_n$. The main consequence of this theorem is that any group $G$ containing distortion elements has no faithful actions as rational affine interval exchange transformations. Moreover, if $G$ is almost-simple, such actions have finite image.

\medskip

\noindent {\bf This paper is organized as follows: }

\smallskip

\begin{itemize}
\item In Section 2,  definitions and basic facts are given.
\end{itemize}

\smallskip

\noindent In sections 3 to 8, we establish  propositions that  play an essential role in the proof of Theorem \ref{TH}, that will be given in section 9.

\begin{itemize}
\item  In section 3, we prove that elements of  $\mathcal A$ with semi-hyperbolic periodic orbits are undistorted.

\item  In section 4, it is shown that given $f\in \mathcal A$, the sequence whose general term is the number of break points of the iterate $f^n $ of $f$ (for simplicity, this sequence will be called ``number of  break points of  $f^n$ '' and it will be denoted by $\# BP(f^n)$) is either bounded or growths linearly. As a consequence, for any distortion element the number of break points of  $f^n $ is  bounded.
\end{itemize}

\smallskip

\noindent In following sections, we study  $f\in \mathcal A$ without semi-hyperbolic periodic orbit and with bounded number of break points of $f^n $.

\begin{itemize}

\item In section 5, Theorem \ref{LI} (Extended ''Alternate Version of Li's Theorem''),  we establish that such an $f$ has an iterate that is conjugate by an element of $\mathcal E$  to a product of restricted PL-homeomorphisms $f_i$ such that numbers of break points of $f_i^n$ are bounded.

\item  For such a PL-homeomorphism, in section 6, we apply results of Minakawa \cite{Min} to prove that it is  PL-conjugate to a PL-homeomorphism $B$ with at most two distinct slopes.

\item  In section 7, under the additional assumption that $f$ is distorted, we derive that $B$ is a rotation, by showing  that the slopes are $1$.

\smallskip

\item  In section 8, Theorem \ref{TH} is proved.

\item  Section 9 is devoted to prove applications of Theorem \ref{TH}: Theorems \ref{BS}, \ref{NIL} and \ref{ratio}.

\end{itemize}

{\bf Acknowledgments.} We are deeply indebted to Andres Navas for bringing this problem to our attention, for allowing us to resume work on the unpublished manuscript\footnote{{\sc Liousse, I. \&  Navas, A.} {\it Distortion elements in $PL_+(S^1)$} (2008).} and for several stimulating discussions during this work.

We  gratefully acknowledge several fruitful discussions with Ignacio Monteverde who also pointed out an error in a preliminary version of this work.

\section{Preliminaries.}

\subsection{Affine Interval Exchange Transformations of  $I=[0,1)$.}

\smallskip
\begin{defi} \

--- A bijection $f$ of  $[0,1)$ is an {\bf Affine Interval Exchange Transformation} (or {\bf
AIET})  of  $[0,1)$ if there exists a finite  subdivision
$0=a_0<a_1<....<a_p=1$ of $[0,1)$ such that for all
$i=0,...,p-1$, one has  $f(x) = \lambda_i x + \beta_i$ for $x\in [a_i,
a_{i+1})$, where $\lambda_i \in \Bbb{R^*_+}$ and  $\beta_i \in
\mathbb{R}$.

\smallskip

--- A {\bf break point} is either the initial point $0$ or a discontinuity of $f$ or  a discontinuity of $Df$, the  derivative  of  $f$.

\smallskip

--- The set of break points of $f$ is  denoted $\mathbf {BP}(f)$; it can be decomposed as the  union of $\mathbf {BP_0}(f)$, the set consisting of $0$ and the discontinuities  of $f$ and $\mathbf {BP_1}(f)$, the set of $0$ and  the discontinuities  of $Df$.

\smallskip

---  We define $\Delta_f(x)= f_+(x)-f_-(x)$, if  $x\in (0,1)$ and  $\Delta_f(0) =  f_+(0)-f_-(1)$, where $f_+(a)=\displaystyle \lim_{x\rightarrow a_+} f(x)=f(a)$
and $f_-(a)=\displaystyle \lim_{x\rightarrow a_-}f(x)$.

\smallskip

--- The $\lambda_i$'s  are the {\bf   slopes } of $f$.

\smallskip

---  The {\bf jump} of $f$ at $x$ is defined by $\sigma_f(x)= \frac{D_+f(x)}{D_-f(x)}$, if  $x\in (0,1)$ and $\sigma_f(0)= \frac{D_+f(0)}{D_-f(1)}$, where $Df_+(a)=\displaystyle \lim_{x\rightarrow a_+} \frac{f(x)-f_+(a)}{x-a}$
and $Df_-(a)=\displaystyle \lim_{x\rightarrow a_-} \frac{f(x)-f_-(a)}{x-a}$.

--- The sets of slopes and jumps of $f$ are denoted respectively by $\Lambda(f)$ and $\sigma(f)$.

\medskip

--- An AIET $f$ of $[0,1)$ is called an {\bf IET } if $\Lambda(f)=\{1\}$.

\end{defi}

\begin{defi} \

--- We denote by $\mathbf{\mathcal A}$ the group consisting of all AIETs of  $[0,1)$.

--- We denote by $\mathbf{\mathcal E}$ the group consisting of all IETs of  $[0,1)$.

\end{defi}

\begin{rema}\label{homeo} \

A homeomorphism $f$  of the circle  $\mathbb S^ 1 = [0, 1]/(0 = 1)$  can be seen as the  bijection of $[0, 1)$ defined by $ x  \mapsto f(x) \   (mod \    1)$.

\end{rema}

\smallskip

\begin{defi} \

When this bijection is an  AIET, $f$ is called  a {\bf PL-homeomorphism} of $[0, 1)$, even if it may not be continuous at eventually one point of $(0, 1)$.
\end{defi}

\noindent In what follows,  PL-homeomorphisms of $\mathbb S^ 1 $ will be seen as AIETs  of  $[0, 1)$. For example, the circle rotation by $\alpha$  is viewed as the element of $\mathcal A$, with break points $0$ and $1-\alpha$,  given by $f(x) =x+\alpha$ if $x\in [0,1-\alpha)$ and  $f(x) =x+\alpha-1$ if $x\in [1-\alpha,1)$ which is called  {\bf rotation} of $[0, 1)$.

\begin{defi}\label{RR} \

An IET $f$ is called a {\bf  restricted rotation} if there exists some interval $I=[a,b)\subset [0,1)$ such that the support of $f$ is $I$ and $f(x)= x+\delta$ if $x \in [a,b-\delta)$ and  $f(x)= x+\delta -b+a$ if $x \in [b-\delta,b)$, where $\delta \in \mathbb R$, $0<\delta<b-a$.

\medskip

An AIET, $f$ is called  a {\bf restricted PL-homeomorphism}  if there exists some interval $I=[a,b)\subset [0,1)$ such that the support of $f$ is $I$ and $\# (BP_0(f)\cap (a,b) )\leq 1$.
\end{defi}

\medskip

Here are some elementary properties of the sets of break points.

\begin{proper}\label{compBP} Let $f$ and $g$ be elements of $\mathcal A$.

--  $BP(f^ {-1})= f(BP(f)$,

--  $BP(f\circ g) \subset BP(g) \cup  g^ {-1}(BP(f))$,

--  $BP(f^n) \subset BP(f) \cup  f^ {-1}(BP(f)) \cup ... \cup f^ {-(n-1)}(BP(f)) $, for all integer $n\geq 0$.

These still hold for $ BP_0(f)$.
\end{proper}

Unfortunately, such formulas do not hold for $BP_1(f)$, this is due to the following

\begin{proper}\label{saut}
Let $f$, $g$ in $\mathcal A$ and $x\in [0,1)$, one has

$${\displaystyle \sigma}_{f\circ g} (x) = \frac { Df_+(g_+(x)) Dg_+ (x)}{ Df_-(g_-(x)) Dg_- (x)}=\frac { Df_+(g_+(x))}{ Df_-(g_-(x))}\times {\displaystyle \sigma}_{g} (x).$$

If $x \notin BP_0(g)$ then  ${\displaystyle \sigma}_{f\circ g} (x)= {\displaystyle \sigma}_{f} (g(x))\times {\displaystyle \sigma}_{g} (x)$.

For  $x \in BP_0(g)$,  ${\displaystyle \sigma}_{f\circ g} (x)\not= {\displaystyle \sigma}_{f} (g(x))\times {\displaystyle \sigma}_{g} (x)$,
in general.
\end{proper}

\begin{proper} \label{sautPL}
Let $f \in \mathcal A$ and  $g$ is a PL-homeomorphism, then

\quad $\forall x\in [0,1), {\displaystyle \sigma}_{f\circ g} (x)= {\displaystyle \sigma}_{f} (g(x))\times {\displaystyle \sigma}_{g} (x)$ and thereby  $BP_1({f\circ g})\subset g^{-1}(BP_1({f})) \cup BP_1(g)$.

\end{proper}
Indeed, if $g$ is a PL-homeomorphism and  $g_+(x)\not=g_-(x)$, then $g_+(x)=0$ and $g_-(x)=1$ therefore $\frac { Df_+(g_+(x))}{ Df_-(g_-(x))}={\displaystyle \sigma}_{f} (0)$.

\subsection{Interesting subgroups of $\mathcal A$.}

Numerous generalizations of Thompson's groups have been
defined and  studied: we recall, for example,
the works of Bieri-Strebel \cite{BS} and Higman \cite{Hi}.

\begin{defi}
Let $\Lambda \subset \mathbb R^{+*}$ be a multiplicative subgroup
and $A \subset \mathbb R$ be an additive subgroup, invariant by
multiplication by elements of $\Lambda $ and such that $1\in A$.
\smallskip

We define $V_{\Lambda,A}$ as the subgroup of $\mathcal A$
consisting of elements with slopes in $\Lambda$, break points and their images in $A$, and $\mathbf{\mathcal E}_{ A}$  as the subgroup of $\mathcal E$
consisting of elements with break points in $A$, in fact $\mathbf{\mathcal E}_{ A}=\mathbf{\mathcal E}\cap V_{\Lambda,A}$.

 The subgroup  ${ A}_{\mathbb Q}$ of rational AIETs is $V_{\mathbb Q_{>0}, \mathbb Q}$.
\end{defi}

\smallskip

\begin{defi}

Let $n$ be a positive integer, the group  $V_{<n>,\  \mathbb Z [{1}/{n}]}$ is denoted by $V_n$ and called  a {\bf Higman-Thompson's group}.
  \end{defi}

 Note that the classical Thompson's group $V$ arises as $V_{2}$.

 \smallskip

 Among many things,  Higman (see \cite{Hi}, \cite{Ro}) proved that $V_n$ is finitely presented and satisfies the conclusions of Theorems \ref{BS} and \ref{NIL}.

\begin{rema} As established in \cite{DFG}, Thompson's groups $F$, $T$ and $V$ are not subgroups of $\mathcal E$.\end{rema}

\begin{rema} Similarly, one can define AIETs of any interval $[a,b) \subset \mathbb R$. In anticipation of a more general use of results from \cite{Li}, \cite{No}, \cite{Min} stated for $[0,1)$, we note that there exists a unique direct affine map sending $[a,b)$ onto $[0,1)$. Thus any AIET of $[a,b)$ is affinely conjugate to an AIET of $[0,1)$ with the same slopes set. However, arithmetic properties of break points might not be preserved. \end{rema}

\begin{defi} Let $I=[a,b)$, We will denote by $\mathcal A (I)$ [resp. $\mathcal E (I)$] the group consisting of AIETs [resp. IETs] of $I$. \end{defi}

\subsection{Distortion.}

\medskip

\begin{defi} \

Let  $\Gamma$ be a  finitely generated  group and  $S=\{s_1,...,s_r\}$ be a finite generating set of $\Gamma$.

\smallskip

The smallest integer  $l$ such that $g=s_{i_1}^ {\epsilon_1}... s_{i_l}^ {\epsilon_l}$, with  ${\epsilon_j}\in \{ -1,1 \}$ is called the  {\bf length} of  $g$ relatively to $S$ and denoted by $l_S(g)$.  \end{defi}

We set $l_S(e)=0$. The function $l_S : \Gamma \rightarrow \mathbb N$ is invariant by taking inverse and satisfies: $l_S(gh)\leq l_S(g) +l_S(h)$. In particular, for all  $g$ in  $\Gamma$, the sequence  $l_S(g^ n)$ is sub-additive, thus the sequence  $\frac{l_S(g^ n)}{n}$ converges. This  leads to

\begin{defi} \

We say that $g$ is {\bf distorted} (or {\bf  of distortion}) in $\Gamma=<S>$ \ if  $g$ has infinite order and $\lim\limits_{n\rightarrow +\infty} \frac{l_S(g^ n)}{n}= 0$.
 \end{defi}

\begin{rema} \

The property of being distorted does not depend on the generating set.

\end{rema}

\begin{defi}

More generally,  if  $G$ is not finitely generated, an element
$g$ of  $G$  is said to be {\bf   distorted} in $G$ if it is a distortion element in some finitely generated subgroup of $G$.
\end{defi}

\begin{propri} \
\begin{itemize}
 \item The following properties are equivalent
\begin{enumerate}
 \item $g\in G$ is distorted in $G$,
 \item $\exists N\in \mathbb Z^*$ : $g^N$ is distorted in $G$,
 \item $\forall N\in \mathbb Z^*$ : $g^N$ is distorted in $G$.
\end{enumerate}

\item If $\Phi : \Gamma \rightarrow G$ is a morphism and  $g\in \Gamma$ is distorted in  $\Gamma$  then its image
$\Phi(g)\in G$ is either of finite order or distorted in  $G$.
\end{itemize} \end{propri}

\section{Semi-hyperbolicity prevents distortion.}

\begin{defi}
Let  $f\in \mathcal A$,  we say that  $p$ is {\bf a semi-hyperbolic periodic point} of period $l$, if either:

 --- $p$ is not a break point of $f^l$,  $f^l(p) = p$ and  $Df^l \not =1$  ({\bf hyperbolic}) or

 --- $p$ is  a break point of $f^l$, $f^l_+(p) = p$ and $Df_+^l(p) \not =1$ or  $f^l_-(p) =p$ and $Df_-^l (p)\not =1$ ({\bf virtual}).

\end{defi}

\begin{prop}\label{hypdist}
If $f\in \mathcal A$  has a  semi-hyperbolic periodic point then   $f$ is undistorted in $\mathcal A$.
\end{prop}

\noindent{\bf Proof.} Let  $p$ be a semi-hyperbolic periodic point of $f$. W.l.o.g,  we can suppose that $f_+(p)=p$ and the right derivative of $f$ at  $p$ : $Df_{+} (p)=\lambda \not= 1$. For clarity, $Df_+$ will be denoted by $D_+ f$.

\medskip

By absurd, suppose that $f$ is distorted  in a subgroup $G$ of  $\mathcal A$ generated by  $S=\{g_1,...,g_s\}$. Then $f$ can be written as  $f^n =g_{i_{l_n}}... g_{i_1}$  with  $\displaystyle \lim_{n \rightarrow +\infty} \frac{l_n}{n}=0$.

\smallskip

We have:  $\displaystyle D_+f^n (p) = D_+g_{i_{l_n}} (p_{l_{n}}) \  ... \  D _+g_{i_1} (p_1)$, where  $p_1=p$ and $p_j=  g_{i_{j-1}} \ ... \  g_{i_1}(p)$, for $j=2,...,l_n$.

\medskip

Then $(\inf D_+g_i )^{l_n} \leq D_+f^n(p) \leq (\sup D_+g_i) ^{l_n}$ and

$$\frac{l_n \log( \inf D_+g_i )}{n} \leq \frac {\log (D_+f^n)(p)}{n} \leq\frac{ l_n \log (\sup D_+g_i)}{n},$$
 where $\displaystyle \inf D_+g_i = \infi_{i,x} D_+g_i (x )$ and  $\displaystyle \sup D_+g_i = \supr _{i,x} D_+g_i (x )$.
\medskip

As $f$ is distorted, one has $\displaystyle \lim_{n \rightarrow +\infty}  \frac {\log (D_+f^n)(p)}{n}=0$.

\smallskip

On the other hand, since $p$ is a fix point of $f$, $\displaystyle D_{+} f^n (p)=\lambda ^n$ and then

\noindent $\displaystyle \lim_{n \rightarrow +\infty}  \frac {\log (D_+f^n)(p)}{n}= \log \lambda \not=0$, this is a contradiction.\hfill $\square$

\section{Alternative for the growth of the number of break points.}

Recall that $BP(f)$ the set of break points of $f$ is the union of the two following sets:

$BP_1(f)=\{a \in [0,1) : \sigma_f(a) \not= 1\}\cup \{0\}$ and  $BP_0(f)=\{a \in [0,1) :  \Delta_f(a)
\not= 0\}\cup \{0\}$.

\smallskip

We denote by $\# BP_* (f)$ the cardinality of $ BP_* (f)$.

\medskip

According to Property \ref{compBP}, one has :
$$\#BP(f^n) \leq \#BP(f)\times n {\text {\ and \ } } \#BP_0(f^n) \leq \#BP_0(f)\times n.$$

 \begin{prop} \label{growth} If $f\in \mathcal A$, then either

 \begin{itemize}

 \item $\# BP(f^n)$ has linear growth or

 \item $\# BP(f^n)$ is bounded.

\end{itemize}

\end{prop}

\noindent{\bf Proof.} Let $a\in BP(f)$.

\smallskip

\begin{proper}   If $a$ is a $f$-periodic point, then for all integer $n$, the set  $ BP(f^n)\cap \mathcal O_f(a)$ is finite and has  cardinality less or equal than the period of $a$.
\end{proper}

\begin{proper}\label{initialBP}  \

 \begin{enumerate}
 \item If $a$ is  not a $f$-periodic point, then there exists a segment $S_a$ of the orbit of $a$: $$\{ b=f^ {-p}(a), f^ {-p+1}(a),..., a, ..., f^ l(a)=c \}, $$
such that $b$ and $c$ belong to $BP(f)$  and for all $k\in \mathbb N^ *$, $f^  {-k}(b) \notin BP(f)$ and $f^  {k}(c) \notin BP(f)$. Such a break point $b$ is called  {\bf initial} break point. Therefore:
\item If $a$ is an  initial break of $f$  point then $S_a=\{ \ a,\  f(a), \ ...\ , \   f^{N_a}(a)\ \}$  and

\noindent  $\mathcal O_f(a)\cap BP (f^n) \subset \{f^{-(n-1)}(a), ..., f^{N_a}(a) \}$, for all integer $n\geq 0$, in particular:

\noindent (a)- $f^{-k}(a)\notin BP(f^m)$, for all integers $k\geq m\geq 0$,

\noindent (b)-  $f ^{\ k}\ (a)\notin BP(f^m)$, for all integers $m\geq 0$, $k>N_a$  and

\noindent (c)-  $f^{-k}(a) \notin  BP(f^{-p})$, for all integers $p\geq 0$, $k\geq 0$.
\end{enumerate}
\end{proper}

 \smallskip

 Indeed, because $\# BP(f)$ is finite, if (1) does not hold  then  there would exist some $d\in BP(f)$, $m_1$ and $m_2$ distinct integers such that  $d=f^ {m_1}(a)=f^ {m_2}(a)$,  which contradicts the non periodicity of $a$. We  derive the second item from Property \ref{compBP}.

%%%%%%%%%%%%%%%%%%%%%%%%%%%%%%%%%%%%%%%%%%%%%%%%%%%%%%%%%%%%%%%%%%%%%%%%%%%

\medskip

\noindent {\bf Step 1: Alternative for $\# BP_0(f^n)$}.

\medskip

Let $a\in BP_0 (f)$ be  a non periodic initial break point and $S_a=\{a,f(a), ...,f^{N_a}(a) \}$; for simplicity of notation, we set $N=N_a$.

\smallskip

\begin{clai} \label{Delta} \

\begin{itemize}

\item If  $\Delta_{f^{N+1}}(a) = 0  $ then for all integer  $l\geq 1$,   $\Delta_{f^{N+l}}(a) = 0  $.

\item  If  $\Delta_{f^{N+1}}(a) \not= 0  $ then for all  integer $l\geq 1$, $\Delta_{f^{N+l}}(a) \not= 0$.
\end{itemize}
\end{clai}

\smallskip

Indeed, let $l\geq 1$, $f_+^{N+l} (a)= f^{N+l} (a)= f^{l-1} (f^{N+1}(a))$ and $f_-^{N+l} (a)= f_-^{l-1} (f_-^{N+1}(a))$.

\smallskip

-- If $\Delta_{f^{N+1}}(a) = 0  $ then $ \Delta_{f^{N+l}}(a)= f^{l-1} (f^{N+1}(a))-f_-^{l-1} (f^{N+1} (a))=0$, since $f^{l-1}$ is continuous at  $f^{N+1} (a)$ by Property \ref{initialBP} (2b).

-- If $\Delta_{f^{N+1}}(a) \not= 0  $, as $f^{l-1}$ is continuous at  $f^{N+1} (a)$, the point $f^{l-1} (f^{N+1}(a))$ can not be equal to $f_-^{l-1} (c)$, for some $c\not=f^{N+1} (a)$. It follows that $f^{N+l}(a)= f^{l-1} (f^{N+1}(a)) \not=f_-^{l-1} (f_-^{N+1}(a))=f_-^{N+l}(a)$,  since $f^{N+1}(a) \not=f_-^{N+1}(a)$.

\bigskip

We turn now to the proof of Step 1, estimating $\#BP_0(f^n)$ for a given   positive integer $n$.

 By Property \ref{initialBP} (2), $\mathcal O_f(a)\cap BP (f^n) \subset \{\ f^{-(n-1)}(a),\  ... \  a , \  ... \ , f^{N}(a) \ \}$.

\smallskip

Let us compute  $\Delta_{f^{n}} (f^{-k}(a)) $, for $0\leq k\leq n-1$, we get:
$$\Delta_{f^{n}} (f^{-k}(a)) = f^{n} _+ (f^{-k}(a))-f^{n}_- (f^{-k}(a)) = f^{n-k}(a)-f^{n-k}_-(f^{k}_-(f^{-k}(a)).$$
Moreover, for all $k\geq 0$, one has $f^{k}_-(f^{-k}(a))= a$,
according to Property \ref{initialBP} (2a).

\medskip

Finally,  $ \Delta_{f^{n}} (f^{-k}(a)) =  f^{n-k}(a)-f^{n-k}_-(a)= \Delta_{f^{n-k}}(a)$.

\bigskip

Summarizing, we have:
\begin{lemm} \label{delta2}
Let $a$ be an initial break point and $n$ be a positive integer.
$$ \Delta_{f^{n}} (f^{-k}(a))=\Delta_{f^{n-k}}(a){\text {, for }} 0\leq k\leq n-1.$$
\end{lemm}

Combining previous Lemma and Claim \ref{Delta}, we deduce that:
\begin{itemize}

\item  If  $\Delta_{f^{N+1}}(a) \not= 0  $ then   $ \Delta_{f^{n}} (f^{-k}(a)) \not= 0  $, for all  $n-k >N$.

 Hence $\# (BP_0(f^n)\cap \mathcal O_f(a)) \geq n-N$.

\smallskip

\item  If  $\Delta_{f^{N+1}}(a) = 0  $ then $ \Delta_{f^{n} }(f^{-k}(a)) =  0$,  for all  $n-k >N$.

Hence  $\# (BP_0(f^n)\cap \mathcal O_f(a)) \leq N + \left( n - (n-N) \right)= 2N$.
\end{itemize}

\medskip

\noindent {\bf Conclusion 1.}

\begin{itemize}

\item  If exists $a\in BP_0(f)$ non periodic initial break point such that $\Delta_{f^{N_a+1}}(a) \not= 0$, then
$$\#BP_0(f)\times n\geq \# BP_0(f^n)\geq n-N_a.$$  That is $\# BP_0(f^n)$ has linear growth.

\medskip

\item  If for all $a\in BP_0(f)$ non periodic initial break point, $\Delta_{f^{N_a+1}}(a) = 0$,
 then $$ \displaystyle BP_0(f^n) \leq \sum_{a\in A }  2N_a +  \sum_{a\in B} period(a),$$ where  $A=\{a\in BP_0(f) {\text { non periodic initial}} \}$ and $B= \{a\in BP_0(f) {\text { periodic } } \} $.

 \smallskip

  That is $\# BP_0(f^n)$ is bounded.

 \end{itemize}

\bigskip

\noindent {\bf Step 2: Alternative for $\# BP(f^n)$}.

\smallskip

If $\# BP_0(f^n)$ is not bounded then $\# BP(f^n)$ has linear growth, by Step 1.

\medskip

Now suppose that $\# BP_0(f^n)$ is bounded.

Let $a\in BP(f)$ be a non periodic initial break point and $S_a=\{a,f(a), ...,f^{N}(a) \}$ be the segment  containing all the break points of $f$ in the orbit of $a$.

\smallskip

Let $n\geq N+1$, recall that $BP(f^n)\cap \mathcal O_f(a)\subset \{f^ {-(n-1)}(a),...,a,..., f^N(a)\}$.

\medskip

Let us compute the jump of $f^{n}$ at the point $f^{-k}(a)$ for $k\geq 0$ and  $n-1-k>N$, that is for $0\leq k < n-1-N$.

Iterating the composition formula given in Property \ref{saut}, we get :

 $${\displaystyle\sigma}_ {f^{n}}(f^{-k}(a)) =
 \frac{ Df_+(f_+^{n-1}(f^{-k}(a))) \ ... \ Df_+(f_+^{k}(f^{-k}(a)) ) \ ... \  Df_+(f^{-k}(a))} {Df_-(f_-^{n-1}(f^{-k}(a))) \ ... \
  Df_-(f_-^{k}(f^{-k}(a) )) \ ... \  Df_-(f^{-k}(a))} .$$

According to Property \ref{initialBP} (2), $ f^{-k}(a) =f_-^{-k}(a)$ and therefore $f_-^l(f^{-k}(a))=f_-^l(f_-^{-k}(a)) =f_-^{l-k}(a)$, for any integer $l\geq 0$;  in addition, if $l\leq k$ then $f_-^l(f^{-k}(a))= f^{l-k}(a)$.

\medskip

Therefore, noting that  $n-1-k > N$, we get: ${\displaystyle\sigma}_ {f^{n}}(f^{-k}(a)) = $
 $$\frac{ Df_+(f_+^{n-1-k}(a) ) \ ... \ Df_+(f_+^{N+1}(a) )}{Df_-(f_-^{n-1-k}(a) ) \ ... \
  Df_-(f_-^{N+1}(a) )} \times
   \frac{ Df_+(f_+^{N}(a) ) \ ... \ Df_+(a)}{Df_-(f_-^{N}(a) ) \ ... \ Df_-(a)}\times
  \frac{Df_+(f^{-1}(a))\ ... \ Df_+(f^{-k}(a))}{Df_-(f^{-1}(a)) \ ... \ Df_-(f^{-k}(a))}.$$

As $BP_1(f)\cap \mathcal O_f(a)\subset S_a$, the third ratio is trivial.

\smallskip

Since $\# BP_0(f^n)$ is bounded, by Conclusion 1 and Claim \ref{Delta}, for all $m\geq N+1$ the point $f^m(a)=f_+^m(a)=f_-^m(a)$  and does not belong to $BP_1(f)$ (by Property \ref{initialBP} (2)). Thus, the first fraction is also trivial. Finally,  $${\displaystyle\sigma}_ {f^{n}}(f^{-k}(a)) =
   \frac{ Df_+(f_+^{N}(a) ) \ ... \ Df_+(a)}{Df_-(f_-^{N}(a) ) \  ... \  Df_-(a)}=:\Pi_a$$

Note that  this formula also holds for $n-1-k =N$, since the first fraction does not appear in ${\displaystyle\sigma}_ {f^{n}}(f^{-k}(a))$.

\medskip
Therefore, the following alternative holds.

\begin{itemize}
\item If  $\Pi_a\not=1  $ then for all $k$ integer such that $0\leq k\leq n-1-N$, $f^{-k}(a)\in BP_1(f^n)$ and $\# (BP_1(f^n)\cap \mathcal O_f(a)) \geq n-N$.

\item  If  $\Pi_a =1  $ then for all $k$ integer such that $0\leq k\leq n-1-N$, $f^{-k}(a)\notin BP_1(f^n)$ and $\# (BP_1(f^n)\cap \mathcal O_f(a)) \leq N +( n- (n-N)) = 2N$.

\end{itemize}

\medskip

\noindent {\bf Conclusion 2.}

\begin{itemize}
\item If exists $a\in BP_1(f)$ non periodic initial break point such that $ \Pi_a  \not= 1$, then
$$\#BP(f)\times n\geq \# BP_1(f^n)\geq n-N_a.$$ Hence $\# BP_1(f^n)$ and  $\# BP(f^n) $  have linear growth.

\medskip

\item If for all $a\in BP_1(f)$ non periodic initial break point $ \Pi_a = 1$, then
 $$ \displaystyle \# BP_1(f^n) \leq \sum_{a\in A }  2N_a +  \sum_{a\in B} period(a),$$ where  $A=\{a\in BP_1(f) {\text { non periodic initial}} \}$ and $B= \{a\in BP_1(f) {\text { periodic} } \} $.

 \smallskip

  Hence $\# BP_1(f^n)$ and  $\# BP(f^n) $ are bounded. \end{itemize} \hfill$\square$

\begin{prop} \label{bound} If $f$ is distorted in AIET then  $\# BP(f^n)$ is bounded. \end{prop}

\noindent{\bf Proof.}

By absurd, suppose that  $\# BP(f^n)$ is unbounded and $f$ is distorted  in a subgroup $G$ of  $\mathcal A$ generated by  $S=\{g_1,...,g_s\}$. This means that  $f^n$ can be written as $f^n =g_{i_{l_n}}... g_{i_1}$  with  $\displaystyle \lim_{n \rightarrow +\infty} \frac{l_n}{n}=0$. Therefore, by Property \ref{compBP},  we have:
$$\displaystyle BP(f^n) \subset  BP(g_{i_1}) \cup g_{i_1}^{-1}BP(g_{i_2}) ....\cup ( g_{i_{l_{n-1}}}... g_{i_1} ) ^{-1}BP(g_{i_{l_n}}), {\text { then }}$$
$$\displaystyle \#BP(f^n)\leq   \#BP(g_{i_1}) + \#BP(g_{i_2}) +...+ \#BP(g_{i_{l_n}}) \leq l_n \max \{\#BP(g_{i}), i=1,...s\}$$
 and therefore  $\displaystyle \frac{ l_n }{n} \geq \frac { \#BP(f^n) }{n} (\max \{\#BP(g_{i}), i=1,...s\})^ {-1}$.

Thus $\displaystyle \lim_{n \rightarrow +\infty} \frac{ l_n }{n}  >0$, since ${ \#BP(f^n) }$ has linear growth, according to Proposition \ref{growth}, this is a contradiction. \hfill $\square$

\section{Extended ''Alternative Version of Li's Theorem''.}
The aim of this section is to prove an extended version of the ''Alternate Version of Li's Theorem'' of \cite{No}.

\begin{theo}\label{LI}
 Let  $f$  in $\mathcal A$ without periodic points and  with $\# BP(f^n)$ bounded then there exists an integer $q$, such that $f^q$ is conjugate in $\mathcal E$ to a  product of restricted PL-homeomorphisms of disjoint support that are minimal when restricted to their respective supports.
\end{theo}

\medskip

\begin{defi}
Let $f \in \mathcal A$. We say that $f$ satisfies {\bf pair property} if

\begin{enumerate}
\item $f$ does not have periodic points,
\item  $BP_0(f)=\{ \beta_1, ....\beta_s,\omega_1,....,\omega_s\}$, any pair $ (\beta_i,\omega_i)$ for $i=1,...s$ verifies $f(\beta_i)=\omega_i$ and $ \beta_i \notin BP_0(f^2)$ and

\item the $f$-orbits of $\beta_i$ are disjoint.
\end{enumerate}
\end{defi}

\noindent{\bf Convention.} Eventually re-indexing the $\omega_i$, we suppose that $0=\omega_1<\omega_2<...<\omega_s$.

\medskip

\noindent {\bf Basic Properties}

\smallskip

If $f$ has pair property, then any  associated pair $ (\beta_i,\omega_i)$, for $i=1,...s$, verifies
\begin{enumerate}
\item $\beta_i \in BP_0(f) \setminus BP_0(f^2)$,
\item $\omega_i \in BP_0(f) \cap BP_0(f^{-1})$,

\item pair property is invariant by $C^0$-conjugation.

\item if $f$ has pair property with associated pairs $ (\beta_i,\omega_i)$ for $i=1,...s$ then, for any $n\in \mathbb N$,
 $f^n$ has pair property with associated pairs $ (f^{-n} (\omega_i),\omega_i)$ for $i=1,...s$.

\end{enumerate}

\smallskip

Using these properties, we get  $f_-(\beta_i)  \in BP_0(f) \cap BP_0(f^{-1}) \cup\{1\}$  so we can give

\begin{defi}
Let $\pi$ be the  permutation  of $\{1,...s\}$ defined by:

either  $j={\pi(i)}$
 in the case that $f_-(\beta_i)= \omega_j$ or  ${\pi(i)}=1$ when $f_-(\beta_i)= 1$.
\end{defi}

Hence, one has $f(\omega_i)=f_-( \omega_{\pi(i)})$, for $i\not=\pi^{-1}(1)$, otherwise, for $i=\pi^{-1}(1)$,  $f(\omega_i)=f_-(1)$.

\begin{defi}

Let $f \in \mathcal A$ with pair property, a pair $ (\beta_i,\omega_i)$ is said {\bf removable} if either:

 (1) $\omega_{\pi(i)}<\omega_i$ or

 (2) $\omega_{\pi(i)}>\omega_i$ and there exists $\omega_j \in (\omega_i, \omega_{\pi(i)})$.

\end{defi}
\begin{lemm} \label{pair}
Let $f \in \mathcal A$ without periodic points and  for which $\#BP_0(f^n)$ is bounded, then there exists an iterate of $f$ that satisfies  pair property.
\end{lemm}

\begin{proof}
According to Section 4, $\displaystyle BP_0(f) \subset  \bigcup_{a \scriptscriptstyle  {\text{  initial break point}} }$ \hskip -1 truecm $ \{ a, ...., f^{N_a}(a)\}$ ,  $\Delta_{f^{N_a+1}}(a) = 0  $  and
\begin{equation}\label{break} \Delta_{f^{n} }(f^{-k}(a)) =  0, {\text { for all  }} k\geq 0, n-k >N_a .\end{equation}

\smallskip

Let $N$ be the maximum of $\{N_a\}$, let  $F=f^{N+1}$, we have that $BP_0(F)\cap \mathcal O_f(a)\subset \{f^ {-N}(a),...,a,..., f^{N_a}(a)\}$.

We note that $\Delta_{F }(f^{-(N+1)+N_a+l}(a)) =  0, $ for $l=1,...,N-N_a+1$ by formula (\ref{break}).
Hence $\displaystyle BP_0(F)\cap \mathcal O_f(a)\subset \bigcup_{l=1}^{N_a} \{f^{-(N+1)+l}(a), f^l(a)\}.$

We claim that any pair $(\beta, \omega)=(f^{-(N+1)+l}(a), f^l(a))$ for $l=1,...,N_a$ satisfies $F(\beta)=\omega$ and $ \beta\notin BP_0(F^2)$.

Indeed,
obviously $F(\beta)=\omega$, we now compute $\displaystyle \Delta_{F^2}(\beta)=\Delta_{f^{2N+2}} (f^{-(N+1)+l } (a)  )=0$ by formula (\ref{break}) with $n= 2N+2, k=N+1 -l$.

It follows that either $f^{-(N+1)+l}(a)\in BP_0(F) $ and $ f^l(a) \in BP_0(F) $ or $f^{-(N+1)+l}(a)\notin BP_0(F) $ and $ f^l(a) \notin BP_0(F) $. So $ BP_0(F) $ is a finite union of pairs of the form
$(f^{-(N+1)+l}(a), f^l(a))$.

Obviously, $F$ satisfies the conditions $(1)$ and $(3)$ of the pair property. \end{proof}

\begin{lemm}\label{remov}

Let $F \in \mathcal A$ with pair property, and  $ (\beta_i,\omega_i)$ a removable pair, then there exists $E \in \mathcal E$ such that $\#BP_0(EFE^{-1})\leq\# BP_0(F)-2$ and $EFE^{-1}$ has also pair property.
Moreover, $BP_0(E) \subset BP_0(F) \cap BP_0(F^{-1})$.
\end{lemm}

\begin{proof} \

-- We begin by considering the case where $\omega_{\pi(i)} <\omega_i$.

\smallskip

Let $E$ be  in $\mathcal E$ with $BP_0(E)=\{0,\omega_{\pi(i)}, \omega_i\}$ and permutation $(1,2,3)\mapsto (1,3,2)$.

According to second item of  Basic Properties,  $BP_0(E) \subset BP_0(F) \cap BP_0(F^{-1})$.

\smallskip
One has that $$BP_0(EFE^{-1})\subset E F^{-1}(BP_0(E))\cup  E (BP_0(F))\cup  BP_0(E^{-1}).$$

As $ F^{-1}(BP_0(E)) \subset BP_0(F) $ and $BP_0(E^{-1})=E(BP_0(E)) \subset  E(BP_0(F) )$, it holds that
$$BP_0(EFE^{-1})\subset E (BP_0(F)) .$$

\smallskip

We prove that $E(\beta_i)$ and $E(\omega_i)$ do not belong to $BP_0(EFE^{-1})$, by computing the right and left values at these points.

\smallskip

\begin{itemize}

\item $EFE^{-1}(E(\beta_i))=EF(\beta_i)=E(\omega_i)=\omega_{\pi(i)}$  and since  $\beta_i\notin BP_0(E)$,

\item  $(EFE^{-1})_-(E(\beta_i))=E_-F_-(\beta_i) =E_-(\omega_{\pi(i)} )=\omega_{\pi(i)} $.

\end{itemize}

\noindent This proves that $E(\beta_i)\notin BP_0(EFE^{-1})$.

\begin{itemize}

\item  $EFE^{-1}(E(\omega_i))=EF(\omega_i)$ and

 \item $(EFE^{-1})_-(E(\omega_i))=(EFE^{-1})_-(\omega_{\pi(i)}))=(EF)_-(\omega_{\pi(i)})$.

Since, $F(\omega_i)=F_-(\omega_{\pi(i)})$ and do not belong to $BP_0(E)$,  $$EFE^{-1}(E(\omega_i))=(EFE^{-1})_-(E(\omega_i)).$$

\end{itemize}
This proves that $E(\omega_i)\notin BP_0(EFE^{-1})$.

\medskip

Finally,  $\#BP_0(EFE^{-1})\leq\# BP_0(F) -2$.

\medskip

We claim that $EFE^{-1}$ has pair property, with associated pairs of the form $(E(\beta_j), E(\omega_j))$.

Indeed, as $(F^{-2} (\omega_i),\omega_i)$ is a pair for $F^2$, it holds that  $BP_0(E) \subset BP_0(F^2) \cap BP_0(F^{-2})$ and same arguments as in the beginning of this proof,   show that $BP_0(EF^2E^{-1})\subset E( BP_0(F^{2})).$

Obviously $EFE^{-1}(E(\beta_j))= E(\omega_j)$.
Suppose, by absurd, that $E(\beta_j)\in BP_0(EF^2E^{-1})$ then $ \beta_j\in  E^{-1}(BP_0(EF^2E^{-1}))\subset  BP_0(F^{2})$, which is a contradiction.

It is clear that $EFE^{-1}$ also satisfies the conditions $(1)$ and $(3)$ of the pair property.

\medskip

-- Now, we consider the case $(\beta_i,\omega_i)$ removable and $\omega_{\pi(i)}> \omega_i$ and
there exists $\omega_j \in (\omega_i, \omega_{\pi(i)})$.

\smallskip

We identify $0$ to $1$ to get a circle and then we cut this circle at the point  $\omega_j$. We are in the previous case. This ends the proof of Lemma \ref{remov}.

\end{proof}

\noindent{\it Proof of Theorem \ref{LI}.}

\smallskip
Let $f\in \mathcal A$ without periodic points and with bounded $ \#  BP_0(f^n)$.

\smallskip

By Lemma \ref{pair}, there exists some integer $N$ such that $f^{ N}$ has the pair property.

Applying  Lemma \ref{remov} a finite number of  times, we get that $G=E_r.....E_1 f^ N E_1^{-1}... E_r^{-1}$ has pair property and no removable pair.

 Since  associated pairs  $(\beta_i,\omega_i)$ of $G$ are not removable,  then $\omega_{\pi(i)} > \omega_i$, for any $i$ and intervals $(\omega_i, \omega_{\pi(i)})$ are pairwise disjoint and disjoint from the last interval $(\omega_s,1)$; note that $s=\pi ^{-1}(1)$, since  by absurd $\omega_s < \omega_{\pi(s)}<1$, a contradiction.

\medskip

We claim that for any $1\leq i <s$, there exists a unique discontinuity point of $G$ in $(\omega_i, \omega_{\pi(i)})$ and  this still holds for  $(\omega_s,1)$.

\medskip

Indeed as $G_+(\omega_{i)})=G_-(\omega_{\pi(i)})$, the interval  $(\omega_i, \omega_{\pi(i)})$ contains at least a point of $BP_0(G)$, similar argument shows that the same holds for $(\omega_s,1)$.

 \smallskip

Pairs are unremovable so this discontinuity point is a $\beta_j$.

Since the number of $\beta$'s is exactly the number of  intervals of the form $(\omega_i, \omega_{\pi(i)})$ or $(\omega_s,1)$, then  $\beta_j$ is unique.

\medskip

Therefore $G([\omega_i, \omega_{\pi(i)}))= [\omega_j, \omega_{\pi(j)})$, since  $G_-(\beta_j)=\omega_{\pi(j)}$ and  $G_+(\beta_j)=\omega_{j}$, if $j \not= \pi ^{-1}(1)=s$. If $ j=s$, then  $G([\omega_i, \omega_{\pi(i)})= [\omega_s ,1 )$.

\smallskip

This implies that $R=\du\limits_{i=1}^{s-1}[\omega_i, \omega_{\pi(i)})\du [\omega_s ,1 )$ is $G$-invariant. We claim that it is $[0,1[$.

\smallskip

Indeed, if not, the complementary of $R$ is a finite union of half open intervals that is $G$-invariant and $G$ is continuous on each interval (since such intervals do not contain $\beta$'s and $\omega$'s the discontinuity points of $G$). Thus, these intervals are periodic which contradicts that $G$ ($f$) does not have periodic points.

\medskip

Moreover,  there exists an iterate of $G$ such that $G^l([\omega_i, \omega_{\pi(i)}))= [\omega_i, \omega_{\pi(i)})$,  $G^l([\omega_s, 1)= [ \omega_s, \omega_1)$ and the restriction of $G^l$ to any $[\omega_i, \omega_{\pi(i)})$,$i=1,..., s-1$ and to  $[\omega_s, 1)$  has just one interior discontinuity point.

\medskip

Finally, $G^l$  is a product of restricted PL-homeomorphisms $\Gamma_i$ with disjoint support and it is conjugated by $E=E_r.....E_1  $ to $f^{lN}$.

As $f$  and then $G^l$ has no periodic points, by Denjoy's Theorem for Class P circle homeomorphisms (see \cite{He}), each $\Gamma_i$ is minimal when restricted to its support.

\smallskip

This ends the proof of Theorem \ref{LI}.\hfill $\square$

\begin{rema}\label{remli}
Note that  endpoints of the supports of the restricted PL-homeomorphisms and discontinuities of the $E_i$'s are in the orbit of $BP_0(f)$.

In particular, if $f\in \mathcal A_{\mathbb{Q}}$ then  endpoints of the supports of the restricted
 PL-homeomorphisms are rational and $E\in  \mathcal A_{\mathbb{Q}}$.

\end{rema}

\section{PL conjugation.}

Next proposition is  due to Minakawa~\cite{Min}.

\begin{prop}\label{PL}

Let $f\in \mathcal A$ a PL-homeomorphism  such that  $\#BP(f^n)$  is bounded. Then there exists a PL-homeomorphism $H\in \mathcal A$ such that $ H\circ f \circ H^{-1}$ is
 an AIET $B=B_{\lambda_1,\lambda_2}$ verifying $\Lambda (B) =\{\lambda_1,\lambda_2\}$ and $BP(B)=\{0, B^ {-1} (0)\}$. In particular, $DB(x) =\lambda_1 $ on $[0, B^ {-1} (0))$ and $DB(x) =\lambda_2 $ on $[B^ {-1} (0),1)$.
\end{prop}

 \begin{rema}\label{Bosh}
The maps $B_{\lambda_1,\lambda_2} $ are PL-homeomorphisms. They were studied in \cite{Bos}. There it was proven that
 $B$ is $C^0$-conjugate to a rotation $R_{\rho}$, by a map
of the form $ \displaystyle x\mapsto \frac{(\omega ^ x-1)}{(\omega -1)}$ for some positive $\omega$ distinct from 1, when $\lambda_1\neq\lambda_2$ and if $\lambda_1 =\lambda_2$ then $B$ is a rotation.
\end{rema}

 \begin{rema}
 We shall give a refinement of  Minakawa's proof which will enable us to preserve arithmetic properties of $f$, this will be explained in Remark \ref{remmina}. An alternative proof using a "PL pair property" can be found in \cite{Lio}.

\end{rema}
\noindent {\bf Proof.}

As $\# BP(f^n)$ is bounded, Conclusion 2 in the proof of Proposition \ref{growth} indicates that there exists a  subset  $\{a_i, i \in \mathcal I \}$ of $ BP_1(f)$ such that  $ BP_1(f)$ is contained in $\bigsqcup_{i\in \mathcal I} {\mathcal S} _i$ with   ${\mathcal S}_i = \{ f^{k}(a_i),\ k=0, ...,N_i\}$ and $$   \Pi_{a_i}=\frac{ Df_+(f_+^{N_i}(a_i) ) .... Df_+(a_i)}{Df_-(f_-^{N_i}(a_i) ) .... Df_-(a_i)}=1, \ \  \forall i\in \mathcal I.$$

Note that  $\frac{ Df_+(f_+^{k}(a_i) )}{Df_-(f_-^{k}(a_i) ) }={\displaystyle\sigma}_ {f}(f^{k}(a_i)) $, since $f_+^{k}(a_i)=f_-^{k}(a_i)$ or $f_+^{k}(a_i)=0$ and $f_-^{k}(a_i)=1$.

Then   $\displaystyle \Pi_{a_i}=\prod_{c\in {\mathcal S}_i} \sigma_f(c) =1$,  for all $i\in \mathcal I$.

\smallskip

Consider a PL homeomorphism $H_f=H$ of $[0,1)$ such that:

  -- the break points of $H$  are the points  $f(a_i)$, \dots , $f^{N_i}(a_i)$, for $ i\in\mathcal  I$,

  --  with associated jumps $\sigma_H ({
    f^{k}(a_i)) = \sigma_{ f^{N+1}} (f^{k}(a_i))}$ for
  $k=1,\dots, N_i$, where $N = \max \{ N_i, i\in \mathcal I\}$.

  Note that  we also have $\sigma_H
  (a_i) = \sigma_{ f^{N+1}}(a_i)=  \prod_{n=0}^N\sigma_f (f^n(a_i)) = 1$.

\smallskip

  At the end of this proof, we indicate a general lemma about existence of PL-homeomorphisms with prescribed break points and slopes. It implies that a necessary and sufficient condition for the existence of
  such a homeomorphism $H$ is that the product of the
  $H$-jumps is trivial, that is

  $$\Pi(f):=\prod_{i\in \mathcal I,0\leq k\leq N_i } \sigma_{ f^{N+1}} (f^{k}(a_i))=1.$$

  -- If $\Pi(f) =1$, then we can define a map $H$ as above and normalize it by setting $H(0) =0$.

  -- If  $\Pi(f) \not=1$, then we add a break point
  $c \notin \{a_i,..., f^{N_i}(a_i)\}$  and require that $\sigma_H (c) = (\Pi (f)  )^{-1}$; we normalize $H$ by setting $H(c)=0$.

\smallskip

 Now, since $f$ and $H$ are PL-homeomorphisms, Property \ref{sautPL}  implies  that

\begin{itemize}

 \item the set $BP_1( H\circ f \circ H^{-1} )$  satisfies

$\begin{aligned}  BP_1(H\circ f \circ H^{-1}) &\subset BP_1( H^{-1})
      \cup  H( BP_1( f )) \cup   H\circ f^{-1} (BP_1(H) ) \\  & \subset \{H(a_i),...,
    H(f^{N_i}(a_i)), i\in  \mathcal I\} \cup \{ H(c), H(f^ {-1} (c))   \},  \end{aligned}$

    \smallskip

  \item for $i\in \mathcal I$, $0\leq k \leq N_i$, the jump of
    $H\circ f \circ H^{-1}$ at $H(f^k(a_i))$ is equal to
  $$
  \begin{aligned}
    \sigma_{H\circ f \circ H^{-1}}\bigl(H(f^k(a_i))\bigr)& =
    \frac{\sigma_{H}(f^{k+1}(a_i)) \times
      \sigma_{f}(f^{k}(a_i))} {\sigma_{H}(f^{k}(a_i))} \\
    &= \frac {\sigma_{f^{N+1}}(f^{k+1}(a_i)) \times
      \sigma_{f}(f^{k}(a_i)) } {
      \sigma_{f^{N+1}}(f^{k}(a_i))}=1,   \end{aligned}
   $$
  \item  $\sigma_{H\circ f \circ H^{-1}}(H(c)) = \Pi(f)$ and

  \item  $\sigma_{H\circ f \circ H^{-1}}\bigl(H(f^ {-1} (c) )\bigr) = \Pi(f)^{-1}$.

  \end{itemize}

   \bigskip

 \noindent{\bf Conclusion.}

 \smallskip

 \begin{itemize}

 \item  If  $\Pi(f) =1$, the AIET $B= H\circ f \circ
  H^{-1}$ has no break of slopes, it is a rotation.

  \item  If  $\Pi(f) \not=1$, the AIET $B= H\circ f \circ
  H^{-1}$ has exactly two break of slopes at $0=H(c)$ and $B^
  {-1} (0) =H(f^ {-1} (c))$, that is $\Lambda(B) = \{ \lambda_1,\lambda_2\}$. \hfill $\square$
   \end{itemize}

\medskip

\begin{lemm}

Given $0=c_0<c_1<...<c_p<1$ points in $[0,1)$ and $\sigma_0, ..., \sigma_p$ positive real numbers such that $\prod\limits_{i=0}^p \sigma_i = 1$, there exists a PL-homemorphism $H$ such that :

-- $BP_0(H)=\{c_0,c_1,...,c_p\}$ and

-- $\sigma_H(c_i) =\sigma_i$, for $i=0, ..., p$.

\end{lemm}

Proof is left to readers, however we indicate some elements of the construction of $H$.

If we denote $\Lambda(H)= \{ \lambda_1, ..., \lambda_{p+1} \}$, one has

-- $\lambda_i = \sigma_0 ... \sigma_{i-1} \lambda_1$ and

-- $\lambda_1= (\vert I_1 \vert +  \sigma_1\vert I_2 \vert + ... +\sigma_1...\sigma_{p} \vert I_{p+1} \vert)^{-1}$ (by computing the total length of $H([0,1))$).

\medskip

In particular, if $c_i$ and $\sigma_i$ are rational numbers then $\lambda_i \in \mathbb Q$. Moreover, we can choose $H$ such that  $H(c_j) \in \mathbb Q$, for some $c_j$ and  then $H\in \mathcal A_{\mathbb Q}$.

\medskip

Note that if $\prod\limits_{i=0}^p \sigma_i \not =1$ such an $H$ does not exist since it should satisfy that $ \lambda_{p+1}= \sigma_0 ... \sigma_{p} \lambda_1$ and $\sigma_{p+1}= \frac{\lambda_{p+1}}{\lambda_{1}}.$

\begin{rema}
\label{remmina}
We have described explicitly the conjugating  PL-homeomorphism $H$, we can deduce that if $f\in \mathcal A_{\mathbb Q}$ then the break points of $H$ and the jumps of $H$  belong to $\mathbb Q$, provided that the point $c$ is chosen in $\mathbb Q$. Therefore, if $f\in \mathcal A_{\mathbb Q}$ then conclusions of Proposition \ref{PL} hold with   $H$ and $B$ belonging to $\mathcal A_{\mathbb Q}$.

\end{rema}

\section{ The case of $\Lambda(B) = \{ \lambda_1,\lambda_2\}$.}

\begin{defi} Let $\alpha_1, ..., \alpha_s$ generating a rank $s$ free abelian multiplicative subgroup $\Lambda$ of $\mathbb R^{+*}$. Therefore, given $\lambda\in\Lambda$, there exists a unique $(n_1,...,n_s) \in \mathbb Z $, such that $\lambda=\alpha_1^{n_1} \ ... \  \alpha_s^{n_s}$ and we  define $\mathcal N_j(\lambda)= n_j$, for  all $j\in \{1,...,s\}$.

\end{defi}

\begin{prop} \label{lambda} \

 Let $B=B_{\lambda_1,\lambda_2}\in \mathcal A$, such that $\Lambda(B) =\{\lambda_1,\lambda_2\} \subset \Lambda$, $BP_1(B)=\{0, a=B^ {-1} (0)\}$ and $B$ is $C^0$-conjugate to an irrational rotation $R_{\rho}$.

 If $(\lambda_1,\lambda_2) \not=(1,1)$ then exist $j\in \{1,...,s\}$ and $x\in [0,1)$ such that $\frac{\mathcal N_j(D_+B^n (x))}{n}\rightarrow \nu \not=0$.

 \end{prop}

 \noindent{\bf Proof.}

Noting that  $B$ satisfies that $DB(x) =\lambda_1 $ on $[0, a)$ and $DB(x) =\lambda_2 $ on $[a,1)$, one has $ D_+B^n (x) = \lambda_1 ^{N_1(x,n)}\lambda_2^{N_2(x,n)}$, where

$\displaystyle N_1(x,n) =\# \{x, f(x), ..., f^{n-1} (x)\} \cap [0, a)=\sum_{k=0}^{n-1} \mathbb I _{[0, a)} (f^{k} (x)) $ and

$\displaystyle N_2 (x,n)=\# \{x, f(x), ..., f^{n-1} (x)\} \cap [a, 1)=\sum_{k=0}^{n-1} \mathbb I _{[a, 1)} (f^{k} (x))$.

 \medskip

The map $B$ has a unique invariant probability measure $\mu$, since it is $C^0$-conjugate to an irrational rotation $R_{\rho}$. More precisely, consider  $h$ such that $h \circ B\circ h^ {-1} =R_\rho$, one has $\mu(A) =\lambda(h(A)$, for all measurable set $A$. In particular,  $\mu ([0,a])= \mu([0,B^ {-1}(0))) =  \lambda([h(0),h\circ B^ {-1}(0)))=\lambda([h(0),R_\rho^{-1}(h(0)))) =(1-\rho)$ and  $\mu ([a,1))=\rho$.

\medskip

 Birkhoff Ergodic Theorem implies that for  $\mu$-almost  every point $x \in [0,1)$, one has   $$ \lim_{n\rightarrow +\infty} \frac{N_1(x,n) }{n} =    \mu ([0,a]) {\text { and } } \displaystyle  \lim_{n\rightarrow +\infty} \frac{N_2(x,n) }{n} = \mu ([a,1]).$$

\medskip
Now, let us write $ \lambda_1$ and $ \lambda_2$ in the basis  $\alpha_1, ..., \alpha_s$ of $\Lambda$: $\lambda_1= \alpha_1^{\beta_1} ... \alpha_s^{\beta_s}$ and $\lambda_2 =\alpha_1^{\delta_1} ... \alpha_s^{\delta_s}$ and  compute the coordinates $\mathcal N_j(D_+B^n (x))$ of $D_+B^n (x)$ in this basis.

\smallskip

As $\displaystyle D_+B^n (x) = \lambda_1 ^{N_1(x,n)}\lambda_2^{N_2(x,n)}=  \alpha_1^{\beta_1.N_1(x,n)+\delta_1.N_2(x,n) } \ ... \  \alpha_s^{\beta_s.N_1(x,n)+\delta_s.N_2(x,n)}$,  one has
$$\mathcal N_j(D_+B^n (x))=
\beta_j.N_1(x,n)+\delta_j.N_2(x,n).$$

 It follows that  $$\frac{\mathcal N_j(D_+B^n (x))}{n} = \beta_j.\frac{N_1(x,n)}{n}+\delta_j.\frac{N_2(x,n)}{n} \rightarrow \beta_j. (1-\rho) +\delta_j\rho = \rho( \delta_j-\beta_j) + \beta_j.$$

\smallskip

Finally, suppose that  $(\lambda_1,\lambda_2) \not=(1,1)$ then necessary $\lambda_1\not=\lambda_2$  and  there exists $j$ such that $\delta_j\not=\beta_j$.  Therefore, $\nu=\rho( \delta_j-\beta_j) + \beta_j\not=0$, as $\rho \notin \mathbb Q$.\hfill $\square$

\section{Proof of Theorem \ref{TH}.}

Let $f$ be distorted in $\mathcal A$, as $f$ has no a semi-hyperbolic periodic point, its periodic points  are not isolated. Using in addition that $BP_0(f)$ is finite, we get that the set $Per(f)$ of $f$-periodic points is the union of a finite collection of half open intervals with endpoints in the orbits of  $BP_0(f)$. Thereby, there exists some positive integer $p$ such that $ Per (f)=Per(f^p) = Fix (f^p)$. It is easy to check that  there exists $S \in \mathcal E$ whose discontinuities are endpoints of connected components of $Per(f)$ and  such that $Fix(S f^p S^{-1})$ is an interval $P=[0,a)$ and  the restriction of $S f^p S^{-1}$ to $M=[a,1)$ has no periodic points.

\medskip

Applying Theorem  \ref{LI} to the restriction of $S f^p S^{-1}$ to $M$,
there exist $q\in \mathbb N^*$ and $E\in \mathcal E$ such that
$ES\circ f^{pq} \circ (ES)^{-1}=\displaystyle \prod\limits_{i=1}^{p} f_i$, where $f_i$ are restricted PL-homeomorphisms with pairwise disjoint support $I_i=[a_i,b_i)$, $f_i\vert_{I_i} $ is minimal  and   $\#BP(f_i^n)$ is bounded (since $f$ is distorted).

\bigskip

Let $i\in \{1 \ ...\ p\}$, by Proposition  \ref{PL} to  $f_i\vert_{I_i} $,   it is conjugate by a PL-homeomorphism $H_i$ of $I_i$ to $B_i$ with $\Lambda(B_i) =\{\lambda_{i,1},\lambda_{i,2} \}$ and  $BP(B_i)=\{a_i, B_i^ {-1} (a_i)\}$. Since  $f_i\vert_{I_i} $ is minimal, $B_i$ also is minimal and according to Remark \ref{Bosh}, $B_i$ is $C^0$-conjugate to an infinite order  rotation $R_{\rho}$ of $I_i$.

\smallskip

Let $H\in \mathcal A$  defined by $H (x) = H_i(x)$, if $x\in I_i$ and $H(x)=x$, if $x\notin \cup I_i$ and let $B= (HES) \circ f^{pq} \circ (HES)^{-1}$.

\smallskip

It is easy to check that $B\vert_{I_i}= B_i$ and  $B$ is distorted in a subgroup $G=<g_1, ..., g_q>$ of $ \mathcal A$, since  $f$ is distorted in $\mathcal A$.

\medskip

Let $\Lambda_G$ be the  free abelian multiplicative subgroup $\Lambda$ of $\mathbb R^{+*}$ generated by $\{\ Dg_k(x),\  x\in [0,1), \ k\in \{1 \ ... \ q\} \ \}$. It has finite rank $s$, we consider a basis $\alpha_1, ..., \alpha_s$ of it.

\medskip

Let $i\in \{1 \ ... \ p\}$, note that  $\mathcal N_j(D_+B^n (y))=\mathcal N_j(D_+B_i^n (y))$, $\forall y \in I_i$. We suppose that  $(\lambda_{i,1},\lambda_{i,2}) \not= (1,1)$.

\medskip

On one hand,  by Proposition \ref{lambda}, there exist $j\in \{1,...,s\}$ and $x \in I_i$ such that $\frac{\mathcal N_j(D_+B^n (x))}{n}\rightarrow \nu \not=0$.

\medskip

On the other hand, since $B$ is  distorted  in $G$,  its iterates $B^n$ can be written \break  $B^n=g_{i_{l_n}} \ ... \ g_{i_1}$  with  $\displaystyle \lim_{n \rightarrow +\infty} \frac{l_n}{n}=0$.

\medskip

Hence $D_+B^n (x) = D_+ g_{i_{l_n}}(x_{l_{n}}) \  ... \  D_+ g_{i_1} (x_1)$, where  $x_m=  g_{i_{m-1}} \ ... \  g_{i_1}(x)$. Then
$$\displaystyle \vert \mathcal N_j(D_+B^n (x)) \vert \leq \sum_{m=1}^{l_n} \vert  \mathcal N_j(D_+ g_{i_m }(x_m) )\vert
\leq {l_n} S,$$

where $S= \max \{\ \vert  \mathcal N_j(D_+ g_{k}(y)) \vert, \ y\in [0,1), \  1\leq k  \leq q \ \}.$

\medskip

Finally,  $\displaystyle \frac{l_n}{n} \geq \frac {\vert \mathcal N_j(D_+B^n (x)) \vert}{nS} \rightarrow \frac {\vert \nu \vert}{S} >0$, this is a contradiction.

\medskip

Consequently, for any $i\in \{1 \ ... \ p\}$, $(\lambda_{i,1},\lambda_{i,2}) = (1,1)$ and thereby $B_i$ is an infinite order rotation of $I_i$.  Thus $B$  is a product of infinite order restricted rotations with pairwise disjoint supports.

\medskip

In conclusion, we have proved that when restricted to $M$, there exists an iterate of $f$ that is conjugate in $\mathcal A$ to a product of infinite order restricted rotations with pairwise disjoint supports. We conclude by noting that  $f_{\vert M^c}=Id_{\vert M^c }$. \hfill $\square$

\section{ Proof of Theorems \ref{BS}, \ref{NIL} and \ref{ratio}.}

\subsection{Proof of Theorem \ref{BS}.} \

Let $a,b$ in $\mathcal A$ such that $ba^m b^{-1}=a^n$ with $m,n$ integers and $\vert m\vert \neq \vert n\vert $. We will prove that $ a$ has finite order.

\smallskip

By absurd, since $a$ is distorted, eventually passing to a power of $a$ and conjugating $a$ and $b$ by an element of $\mathcal A$ we can suppose that $a$ is a product of infinite order restricted rotations $R_{\alpha_i}$ of disjoint supports $I_i$. We denote $\displaystyle a=\prod_{i=1}^{p}(R_{\alpha_i},I_i)$.

The main tool of the proof is the following
\begin{lemm} \label{comm}
Let $a,b \in \mathcal A$ with $a=\prod_{i=1}^{p}(R_{\alpha_i},I_i)$ and  $ba^m b^{-1}=a^n$ then there exists an integer $s$ such that $b^s$ maps $I_i$ to itself preserving the Lebesgue measure on $I_i$, $Leb\vert I_i$.

\end{lemm}
Let $X=\cup I_i$.
In what follows we identify $a$ with its restriction to $X$.

By unique ergodicity of  irrational rotations, one has that ergodic $a^p$-invariant probabilities on $X$ are $Leb \vert I_i$, for all $p\in \mathbb Z$.

As $Supp(ba^mb ^{-1})=b(Supp(a^m))$ and $Supp(a^n)=Supp(a^m)$, then $b(X)=X$ and $b$ is identified to its restriction to $X$.

The image by $b$ of an ergodic $a^m$-invariant measure is an ergodic $a^n$-invariant measure.

Hence, for some permutation $\sigma$, $b_{\star} (Leb \vert I_i)=Leb \vert I_{\sigma(i)}$. Thus there exists an integer  $s$ such that $b^s_{\star} (Leb \vert I_i)=Leb \vert I_{i} $.  \hfill $\square$

\medskip

Spectrum of irrational rotations viewed as IET  are $Sp((R_{\alpha},I, Leb\vert I )=<e^{2i\pi\frac{\alpha}{l}}>$ where $l=\vert I\vert$.

As a consequence of  $b^s a^{m^s} b^{-s}=a^{n^s}$, $b^s\vert I_i$ sends the  generator of $Sp(a^{m^s}, I_i, Leb \vert I_i )$ into a generator of $Sp(a^{n^s}, I_i, Leb \vert I_i)$. Then $$ e^{2i\pi\frac{\alpha_i}{l_i}m^s}=e^{\pm 2i\pi\frac{\alpha_i}{l_i} n^s}.$$

Finally, $\frac{\alpha_i}{l_i}m^s=\pm \frac{\alpha_i}{l_i}n^s $ mod $\mathbb{Z}$.
This is a contradiction since $\frac{\alpha_i}{l_i}\notin \mathbb{Q}$ and  $\vert m\vert \neq \vert n\vert $.

Therefore $a$ has finite order, hence any action of $BS(m,n)$ with  $\vert m\vert \neq \vert n\vert $ by elements of  $\mathcal A$  is not faithful.

\subsection{Proof of Theorem \ref{NIL}.} \

Let $G$ be a  nilpotent subgroup of $\mathcal A$.

Suppose by absurd that  $G$ is either  non abelian torsion free or finitely generated and not virtually abelian.

Since $G$ is nilpotent there exist $u,v \in G$ such that $c=[u,v]$ commutes with $u$ and $v$. Furthermore, we can choose $c$ of infinite order because $G$ is either non abelian torsion free or finitely generated and not virtually abelian. This implies

\begin{claim} \label{cla}For any integers $p$ and $q$, it holds that $$[u^p,v^q]=c^{pq}.$$
\end{claim}

In particular, $c^{n^2}=[u^n,v^n],$ so $c$ is distorted.

 Hence, eventually passing to a power of $c$ and conjugating by an element of $\mathcal A$ we can suppose that $c$ is a product of infinite order restricted rotations $R_{\alpha_i}$ of disjoint supports $I_i$. We denote $\displaystyle c=\prod_{i=1}^{p}(R_{\alpha_i},I_i)$.

\medskip

Applying  Lemma \ref{comm} with $m=n$, $a=c$, and  $b=u$ [resp. $b=v$],  there exist $s_u$ [resp. $s_v$]
such that $u^{s_u}_{\star} (Leb \vert I_i)=Leb \vert I_{i} $ and $v^{s_v}_{\star} (Leb \vert I_i)=Leb \vert I_{i}$.

Therefore $u^{s_u}\vert I_i$ and  $v^{s_v}\vert I_i$ are  IET which commute with the rotation $R_{\alpha_i}$. Finally, by Lemma 5.1 of \cite{No},   $u^{s_u} \vert I_{i} $ and $v^{s_v}\vert I_i$ are rotations so they commute.

	According to claim \ref{cla}, $[u^{s_u},v^{s_v}]=c^{s_u s_v}$, so $ c^{s_u s_v}=Id$ on $I_i$ for any $i=1,...,p$. It follows that $c^{s_u s_v}=Id$. This contradicts that $c$ has infinite order.

\subsection{Proof of Theorem \ref{ratio}.} \

In this section we prove that there is no distortion elements in $\mathcal A_{\mathbb Q}$.

Let $f\in  \mathcal A_{\mathbb Q}$ distorted in $\mathcal A_{\mathbb Q}$, by Theorem \ref{TH}, there exist a positive integer $s$ and $H\in \mathcal A$, $E\in \mathcal E$ and  $S\in \mathcal E$ such that  $(HES) f^{pq} (HES)^{-1}= \prod\limits_{i=1}^{p}(R_{i},I_i)$ is a product of infinite order restricted rotations of disjoint supports $I_i$.

\medskip

We  first check  that  the conjugating maps  $H$, $E$ and  $S$ are in $\mathcal A_{\mathbb Q}$.

 By definition of $S$,  break points of $S$ are endpoints of connected components of $Per(f)$ so belong to the orbit of $ BP_0(f)$, that is contained in $\mathbb Q$.  Therefore $S \in  \mathcal A_{\mathbb Q}$

According to Remark \ref{remli}, $E\in \mathcal A_\mathbb{Q}$ and endpoints of the $I_i$'s are rational. Hence, by  Remark \ref{remmina}, $H\in \mathcal A_{\mathbb Q}$.

\medskip

Therefore  $(HES) f^{pq} (HES)^{-1} \in \mathcal A_{\mathbb Q}$ and then $R_i\in \mathcal A_{\mathbb Q}$. This is a contradiction. \hfill $\square$

\end{document}